\documentclass[twoside]{article}
\usepackage[cp1251]{inputenc}
\usepackage[T2A]{fontenc}
\usepackage{array}
\usepackage{amssymb}
\usepackage{amsmath}
\usepackage{amsthm}
\usepackage{latexsym}
\usepackage{indentfirst}
\usepackage{bm}
\usepackage{enumerate}
\usepackage[dvips]{graphicx}
\usepackage{epsf}
\usepackage{wrapfig}
\usepackage{euscript}
\usepackage{indentfirst}
\usepackage[english,russian]{babel}
\usepackage{textcomp}
\newtheorem{theorem}{Theorem}
\newtheorem{lemma}{Lemma}

\begin{document}
{\selectlanguage{english}
\binoppenalty = 10000 %
\relpenalty   = 10000 %

\pagestyle{headings} \makeatletter
\renewcommand{\@evenhead}{\raisebox{0pt}[\headheight][0pt]{\vbox{\hbox to\textwidth{\thepage\hfill \strut {\small Grigory. K. Olkhovikov}}\hrule}}}
\renewcommand{\@oddhead}{\raisebox{0pt}[\headheight][0pt]{\vbox{\hbox to\textwidth{{On imagination logic}\hfill \strut\thepage}\hrule}}}
\makeatother

\title{An axiomatic system for STIT imagination logic}
\author{Grigory K. Olkhovikov}
\date{}
\maketitle

\begin{quote}
{\bf Abstract.} We formulate a Hilbert-style axiomatic system for
STIT logic of imagination recently proposed by H. Wansing in
\cite{Wansing} and prove its completeness by the method of
canonical models.
\end{quote}

\begin{quote}
\end{quote}

We assume a propositional language with a countably infinite set
$Var$ of propositional variables and the following set of
modalities:

(1) $SA$ understood as `$A$ is settled true'; the dual modality is
$PA$ understood `$A$ is possible'.

(2) $[c]_aA$ understood as `agent $a$ $cstit$-realizes $A$'; the
other action modality, namely, $[d]_aA$ to be read `agent $a$
$dstit$-realizes $A$', is in this setting a defined one with the
following definition: $[c]_aA \wedge \neg SA$.

(3) $I_aA$ understood as `agent $a$ imagines that $A$'.

Among other things, all the agent indices are assumed to stand for
pairwise different agents.

For these modalities we assume the following
`stit-plus-neighborhood' semantics originally defined by H.
Wansing in \cite{Wansing}.

An imagination model is a tuple $\mathcal{M} = \langle Tree, \leq,
Ag, Choice, \{ N_a\mid a \in Ag \}, V\rangle$, where:
\begin{itemize}
\item $Tree$ is a non-empty set of moments, and $\leq$ is a
partial order on $Tree$ such that
$$\forall m_1,m_2\exists m(m \leq m_1 \wedge m \leq m_2),$$
and
$$\forall m_1,m_2,m((m_1 \leq m \wedge m_2 \leq m) \to (m_1 \leq m_2 \vee m_2 \leq m_1)).$$

\item The set $History$ of all histories of $\mathcal{M}$ is then
just a set of all maximal $\leq$-chains in $Tree$. A history $h$
is said to pass through a moment $m$ iff $m \in h$. The set of all
histories passing through $m \in Tree$ is denoted by $H_m$.

\item $Ag$ is a finite set of all agents acting in $Tree$ and is
assumed to be disjoint from all the other items in $\mathcal{M}$.

\item $Choice$ is a function defined on the set $Tree \times Ag$,
such that for an arbitrary $\langle m, a\rangle \in Tree \times
Ag$, we the value of this function, that is to say $Choice(m,a)$
(more commonly denoted $Choice^m_a$) is a partition of $H_m$.  If
$h \in H_m$, then $Choice^m_a(h)$ denotes the element of
$Choice^m_a$, to which $h$ belongs. In the special case when we
have $Choice^m_a = \{ H_m \}$, it is said that the agent $a$ has a
\emph{vacuous} choice at the moment $m$. In our models, $Choice$
is assumed to satisfy the following two restrictions:
\begin{itemize}
\item ``No choice between undivided histories'': for arbitrary
$m\in Tree$, $a \in Ag$, $e\in Choice^m_a$, and $h,h' \in H_m$:
$$(h \in e \wedge \exists m'(m < m' \wedge m' \in h \cap h')) \to h' \in e.$$

\item ``Independence of agents''. If $f$ is a function defined on
$Ag$ such that  $\forall a \in Ag(f(a) \in Choice^a_m)$, then
$\bigcap_{a \in Ag}f(a) \neq \emptyset$.
\end{itemize}

\item The set of moment-history pairs in $\mathcal{M}$, that is to
say, the set
$$MH(\mathcal{M}) = \{ \langle m, h\rangle \mid m \in
Tree,  h \in H_m \}$$

is then to be used as a set of points, where formulas are
evaluated.

\item For every $a \in Ag$, we have $N_a: MH(\mathcal{M}) \to
2^{(2^{MH(\mathcal{M})})}$. $N_a$ is thus a neighborhood function,
defining, for every moment history pair $m/h$ the set of
propositions imagined by the agent $a$ at the moment $m$ in
history $h$.

\item $V$ is an evaluation function for atomic sentences, that is
to say, $V: Var \to 2^{MH(\mathcal{M})}$.
\end{itemize}

The relation of satisfaction of sentences in the above defined
language by moment-history pairs in $\mathcal{M}$ is then defined
inductively as follows:
\begin{align*}
\mathcal{M}, m/h &\vDash p \Leftrightarrow m/h \in V(p), &&\text{
for atomic $p$;}\\
\mathcal{M}, m/h &\vDash (A \wedge B) \Leftrightarrow \mathcal{M},
m/h \vDash A \wedge \mathcal{M}, m/h \vDash B;\\
\mathcal{M}, m/h &\vDash \neg A \Leftrightarrow \mathcal{M}, m/h
\not\vDash A;\\
\mathcal{M}, m/h &\vDash SA \Leftrightarrow \forall h' \in
H_m(\mathcal{M}, m/h' \vDash A);\\
\mathcal{M}, m/h &\vDash [c]_aA \Leftrightarrow \forall h' \in
Choice^m_a(h)(\mathcal{M}, m/h' \vDash A);\\
\mathcal{M}, m/h &\vDash I_aA \Leftrightarrow \forall h' \in
Choice^m_a(h)(\{ m/h \in MH(\mathcal{M}) \mid \mathcal{M}, m/h
\vDash A \} \in N_a(m/h')) \wedge\\
&\qquad\qquad\qquad\qquad\wedge\exists h'' \in H_m((\{ m/h \in
MH(\mathcal{M}) \mid \mathcal{M}, m/h \vDash A \} \notin
N_a(m/h''))).
\end{align*}

For this logic we propose the following axiomatization:

(A0) Propositional tautologies.

(A1) $S$ is an $S5$ modality.

(A2) For every $a \in Ag$, $[c]_a$ is an $S5$ modality.

(A3) $SA \to [c]_aA$ for every $a \in Ag$.

(A4) $(P[c]_{a_1}A_1 \wedge\ldots \wedge P[c]_{a_n}A_n) \to
P([c]_{a_1}A_1 \wedge\ldots \wedge[c]_{a_n}A_n)$, provided that
all the $a_1\ldots a_n$ are pairwise different.

(A5) $I_aA \to ([c]_aI_aA \wedge \neg SI_aA)$ for every $a \in
Ag$.

Rules are as follows:

(R1) Modus ponens.

(R2) From $A$ infer $SA$.

(R3) From $A \leftrightarrow B$ infer $I_aA \leftrightarrow I_aB$
for every $a \in Ag$.

\textbf{Note}. Thus the proposed axiomatization is just the
axiomatization of $dstit$ logic proposed by Ming Xu plus
axiomatization of the logic of $I_a$ as a minimal neighborhood
modal system \textbf{E} plus the special axiom (A5) stating the
action character of the imagination operator. Note also that the
converse of (A5) easily follows from (A2), so that we actually
have a biconditional here.

Our aim now is to get a strong completeness theorem for this
system $L$ with respect to the above semantics, in the following
form: if $\Theta$ is an $L$-consistent set of sentences, then
$\Theta$ has a model in your proposed semantics.

In what follows we will always use `consistency' to mean
`$L$-consistency' and we let $\vdash$ stand for a relation of
$L$-derivability.

In order to get the main theorem, we use the technique of
canonical models, which is an adaptation of the corresponding
techniques for the two respective parts of our system as mentioned
in the Note above. In particular, we draw on \cite[ch. 17]{FF} in
many matters relevant to the purely STIT part of the following
construction.

More precisely, we let $W$ to be the set of all $L$-maxiconsistent
sets of sentences and we denote the members of $W$ as $w$, $w'$,
$w_1$ etc. We set $wRw'$ iff $\{A \mid SA \in w \} \subseteq w'$,
and we set $w\simeq_aw'$ iff $\{A \mid [c]_aA \in w \} \subseteq
w'$. By standard modal logic, (A1) and (A2) ensure that all these
relations are relations of equivalence; moreover, (A3) ensures
that $\simeq_a \subseteq R$ for every $a \in Ag$.

Indeed, let $w\simeq_aw'$ and let $SA \in w$. By (A3) and
maxiconsistency of $w$, we get $[c]_aA \in w$, whence by
$w\simeq_aw'$ we get that $A \in w'$. Since $A$ was arbitrary,
this means that $wRw'$.

In what follows, we will be denoting equivalence classes of $W$
with respect to $R$ by $X$, $X'$, $X_1$, etc. The set of all such
equivalence classes will be denoted by $\Xi$. When restricted to
an arbitrary $X \in \Xi$, the relation $R$ turns into a universal
relation, but relations of the form $\simeq_a$ can remain
non-trivial equivalences breaking $X$ up into several equivalence
classes. We will denote the family of equivalence classes
corresponding to $\simeq_a\upharpoonright X$ by $E(X,a)$.

Among the elements of $W$, we have a special interest in the
maxiconsistent sets extending the following set of formulas:
$$
\Sigma = \{ \neg p \mid p \in Var \} \cup \{ SA \leftrightarrow A
\mid \text{ for arbitrary }A \} \cup \{ [c]_aA \leftrightarrow A
\mid \text{ for arbitrary }A \}.
$$

The following facts are worth noting:

(F1) There exists exactly one element in $W$, which extends
$\Sigma$. We will denote this element by $\mathbf{w}$. Indeed, one
easily sees that $\Sigma$ pre-determines every Boolean formula by
fixing the literals. The modalities $S$ and $[c]_a$ are then just
vacuous in virtue of the definition of $\Sigma$. Finally, every
maxiconsistent set extending $\Sigma$ will have to contain $\neg
I_aA$ for every formula $A$ and every $a \in Ag$. For suppose
otherwise. Then for some $w \in W$ such that $\Sigma \subseteq w$,
for some formula $A$ and for some $a \in Ag$ we will have $I_aA
\in w$. Then, by (A5) and maxiconsistency of $w$ we will get $\neg
SI_aA \in w$. Therefore, by definition of $\Sigma$ and
maxiconsistency of $w$, we will get $\neg I_aA \in w$, which
contradicts the assumption that $w \in W$. Therefore, the
statements with $I_a$-modalities are also fixed for every $w \in
W$, for which $\Sigma \subseteq w$. It is also easy to see that
such a maxiconsistent $w$ extending $\Sigma$ must exist, since
$\Sigma$ itself is obviously consistent\footnote{$\Sigma$ is
satisfiable and thus consistent. Indeed, consider a model
consisting of a single moment, where every agent has a vacuous
choice, every set of imagination neighborhoods is empty and every
variable valuation is empty as well.}

(F2) It follows from the definitions of $\Sigma$ and $R$ that the
$R$-equivalence set containing $\mathbf{w}$, contains $\mathbf{w}$
only. We will denote this equivalence set by $\mathbf{X}$.

We now proceed to the definition of our canonical model. First, we
choose\footnote{We also assume, with the view of the definition of
$\leq$ below, that $0$ is not an element of any element of $\Xi
\cup W$.} an element $0 \notin \Xi \cup W$ and define our set of
moments:
$$Tree = \{ 0 \} \cup \Xi \cup W.$$

We then set the following partial order on $Tree$. For arbitrary
$x,y \in Tree$ we have $x \leq y$ iff $x = y$, or $y \in x$ or $x
= 0$. This allows for a simple description of the set of histories
in our frame. Every history turns out to have the form $h_w =
\langle 0, X, w\rangle$, where $X \in \Xi$ and $w \in X$. Thus,
our set of histories is in one-to-one correspondence with $W$.

Thirdly, we define the choice function. It assigns a vacuous
choice to every agent at every moment $m$, if $m \notin \Xi$. That
is to say, the only choice of every agent at every such moment
will be just the set of all histories passing through this moment.
Otherwise, i.e. for the case when $m = X \in \Xi$, we define the
choice function as follows:
$$Choice^a_X = \{ H \mid \exists e \in E(X,a)(H = \{ h_w \mid w \in e \}) \}.$$

Next, we need to define the imagination neighborhoods. We do this
in the following way. $N_a(m/h) = \emptyset$ for every $a \in Ag$
and every $m \notin \Xi$. For the case when $m = X \in \Xi$, we
need one further auxiliary notion. For every sentence $A$ we set
$Ext(A)$ (read: extension of $A$) to be $\{ X/h_w\mid w\in X
\wedge A\in w \}$ if $A \notin \mathbf{w}$; otherwise we set
$$Ext(A) = \{ X/h_w\mid w\in X \wedge A\in w \} \cup \{ m/h_w\mid m \notin \Xi \wedge m \in h_w \}.$$
Having defined the extensions, we set
$$N_a(X/h_w) = \{ Ext(A) \mid I_aA \in w \}$$
for arbitrary $w \in X \in \Xi$.

Finally, we define the evaluation function for variables in the
following way:
$$V(p) = \{ X/h_w\mid w\in X \in \Xi \wedge p\in w \}.$$

We need to show that the canonical model $\mathcal{M}$ defined
above is the model of our logic. The semantic restrictions are
mostly seen to hold immediately; in particular, the
no-choice-between-undivided-histories restriction holds because we
only have undivided histories at the moment $0$, where only
vacuous choices are allowed. The only exception is the
independence-of-agents restriction, which we treat below.

\begin{lemma}[On Independence]\label{independence}
Let $m \in Tree$ and let $f$ be a function on $Ag$ such that
$\forall a \in Ag(f(a) \in Choice^a_m)$. Then $\bigcap_{a \in
Ag}f(a) \neq \emptyset$.
\end{lemma}

\begin{proof}
If $m \notin \Xi$, then the statement of the Lemma is obvious,
since every agent will have a vacuous choice. We treat the case,
when $m = X \in \Xi$. Consider a function $f$ as described in
Lemma. For every $f(a)$ we fix $e_{f(a)} \in E(X,a)$ such that
$f(a) = \{ h_w \mid w \in e_{f(a)} \}$ and we fix, further, an
arbitrary $w_{f(a)} \in e_{f(a)}$. Since $e_{f(a)}$ is an
$\simeq_a$-equivalence class, there is a set $\Gamma_{f(a)}$ of
sentences of the form $[c]_aA$ shared by all the members of
$e_{f(a)}$ and only those members. Also, since $X$ is an
$R$-equivalence class, there is a set $\Delta$ of sentences of the
form $SA$ shared by all (and only) members of $X$. Consider, then,
the following set of sentences:
$$\Lambda = (\bigcup_{a \in Ag}\Gamma_{f(a)}) \cup \Delta.$$

We claim that $\Lambda$ is consistent. Assume otherwise. In this
case $\Lambda$ contains a finite inconsistent subset. Given that
$S$ and $[c]_a$ are $S5$-modalities, we can assume that this
inconsistent subset has the following form:
$$SB, [c]_{a_1}A_1,\ldots, [c]_{a_n}A_n,$$
where all the $a_1\ldots a_n$ are pairwise different (and
moreover, $Ag = \{ a_1\ldots a_n \}$). We know, further, that for
all $1 \leq i \leq n$ we have $SB, [c]_{a_i}A_i \in w_{f(a_i)}$.
So, choose an arbitrary $w \in X$. For every $1 \leq i \leq n$ we
have $w_{f(a_i)}Rw$, therefore, we must also have $P[c]_{a_i}A_i
\in w$ for every $1 \leq i \leq n$. Indeed, if it were otherwise,
we would have $S\neg [c]_{a_i}A_i \in w$ since $w$ is
maxiconsistent. But then, given that $wRw_{f(a_i)}$, we would have
$\neg [c]_{a_i}A_i \in w_{f(a_i)}$, a contradiction.

Thus, we have in fact that
$$P[c]_{a_1}A_1\wedge\ldots \wedge P[c]_{a_n}A_n \in w,$$
therefore, by (A4), we also have
$$P([c]_{a_1}A_1\wedge\ldots \wedge [c]_{a_n}A_n) \in w.$$
This, in turn, means that the set
$$\{ A \mid SA \in \Delta \} \cup \{ [c]_{a_1}A_1\wedge\ldots \wedge [c]_{a_n}A_n \}$$
is consistent: otherwise, we would have that
$$\{ A \mid SA \in \Delta \} \vdash \neg([c]_{a_1}A_1\wedge\ldots \wedge [c]_{a_n}A_n),$$
and, by standard modal $S5$-reasoning, that
$$\Delta \vdash S\neg([c]_{a_1}A_1\wedge\ldots \wedge [c]_{a_n}A_n),$$
which, given that $w \in X$ and hence $\Delta \subseteq w$, would
mean inconsistency of $w$, a contradiction.

Therefore, we may choose an arbitrary maxiconsistent $w'$
extending $\{ A \mid SA \in \Delta \} \cup \{
[c]_{a_1}A_1\wedge\ldots \wedge [c]_{a_n}A_n \}$, and by the fact
that this set contains $\{ A \mid SA \in \Delta \}$ we know that
$wRw'$ and thus $w' \in X$ and further $SB \in w'$. This means
that our finite subset in fact has a model and is not
inconsistent. Therefore, since the finite set was arbitrary,
$\Lambda$ is consistent as well. Consider, then, an arbitrary
maxiconsistent $w''$ extending $\Lambda$. Since $\Delta \subseteq
w''$, we have $w'' \in X$, and since $\Gamma_{f(a)} \subseteq w''$
for arbitrary $a \in Ag$, we have $w''\simeq_a w_{f(a)}$ for every
such $a$. This means, in turn, that $w'' \in e_{f(a)}$ for every
$a \in Ag$, and so $h_{w''} \in \bigcap_{a \in Ag}f(a) \neq
\emptyset$.
\end{proof}

By now, the only ingredient to be added is the Truth Lemma; we
divide it into two parts as follows.

\begin{lemma}[Truth Lemma 1]\label{truthlemma1}
Let $m \notin \Xi$ and $m \in h$. Then, for any sentence $A$, the
following holds:
$$\mathcal{M}, m/h \vDash A \Leftrightarrow A \in \mathbf{w}.$$
\end{lemma}
\begin{proof}
We use induction on the construction of $A$. If $A = p \in Var$,
then $A \notin \mathbf{w}$, and also $m/h \notin V(A)$, since $m
\notin \Xi$. Therefore, $\mathcal{M}, m/h \not\vDash A$.

The boolean cases are then trivial.

If $A = SB$, then $\mathcal{M}, m/h \vDash A$ iff $\mathcal{M},
m/h' \vDash B$ for every $h'$ such that $m \in h'$ iff $A \in
\mathbf{w}$ by induction hypothesis (since we have proved IH for
arbitrary $h$ going through $m$).

If $A = [c]_aB$, then $\mathcal{M}, m/h \vDash A$ iff
$\mathcal{M}, m/h' \vDash B$ for every $h'$ such that $m \in h'$
and $h' \in Choice^a_m(h)$ iff $A \in \mathbf{w}$ by induction
hypothesis (cf. the commentary on the previous case).

If $A = I_aB$, then $A \notin \mathbf{w}$ by (F1). We also have
$\mathcal{M}, m/h \not\vDash A$, since, given that $m \notin \Xi$,
all the choices at $m$ are vacuous.
\end{proof}

\begin{lemma}[Truth Lemma 2]\label{truthlemma2}
Let $X \in \Xi$ and $w \in X$. Then, for any sentence $A$, the
following holds:
$$\mathcal{M}, X/h_w \vDash A \Leftrightarrow A \in w.$$
\end{lemma}
\begin{proof}
Again, we use induction on the construction of $A$. Atomic case we
have by definition of $V$, and the boolean cases are obvious. We
consider the modal cases.

Let $A = SB$, and assume that $SB \in w$. Then take any $h_{w'}$
passing through $X$. In the context of $\mathcal{M}$ this means
that $w' \in X$, which in turn means that $wRw'$. Therefore, we
have $B \in w'$ and, by induction hypothesis, $\mathcal{M},
X/h_{w'} \vDash B$. Since $h_{w'}$ was arbitrary, this means that
$\mathcal{M}, X/h_w \vDash SB$.

On the other hand, assume that $SB \notin w$. This means that the
set
$$\alpha = \{ C \mid SC \in w \} \cup \{ \neg B \}$$
is consistent. Indeed, otherwise we would have
$$\{ C \mid SC \in w \} \vdash B,$$
and further, by standard $S5$ reasoning
$$\{ SC \mid SC \in w \} \vdash SB,$$
and so, given, maxiconsistency of $w$, we would have $SB \in w$,
contrary to our assumption. Therefore, consider an arbitrary $w'
\in W$ extending $\alpha$. By definition, $w' \in X$, therefore
$h_{w'}$ goes through $X$ and we have, by induction hypothesis,
that $\mathcal{M}, X/h_{w'} \not\vDash B$.

Let $A = [c]_aB$, and let $[c]_aB \in w$. Then take any $h_{w'}$
such that $h_{w'} \in Choice^a_X(h_w)$. In the context of
$\mathcal{M}$ this means that $w\simeq_a w'$. Therefore, we have
$B \in w'$ and, by induction hypothesis, $\mathcal{M}, X/h_{w'}
\vDash B$. Since $h_{w'}$ was arbitrary, this means that
$\mathcal{M}, X/h_w \vDash [c]_aB$.

On the other hand, assume that $[c]_aB \notin w$. This means that
the set
$$\beta = \{ C \mid [c]_aC \in w \} \cup \{ \neg B \}$$
is consistent. Indeed, otherwise we would have
$$\{ C \mid [c]_aC \in w \} \vdash B,$$
and further, by standard $S5$ reasoning
$$\{ [c]_aC \mid [c]_aC \in w \} \vdash [c]_aB,$$
and so, given, maxiconsistency of $w$, we would have $[c]_aB \in
w$, contrary to our assumption. Therefore, consider an arbitrary
$w' \in W$ extending $\beta$. By definition, $w'\simeq_aw$, and
also $w' \in X$ given that $\simeq_a \subseteq R$. Therefore
$h_{w'}$ goes through $X$ and moreover $h_{w'} \in
Choice^a_X(h_w)$. By induction hypothesis, we have that
$\mathcal{M}, X/h_{w'} \not\vDash B$, and so, putting all
together, that $\mathcal{M}, X/h_w \not\vDash [c]_aB$

Let $A = I_aB$. First of all, note that by induction hypothesis
and Lemma \ref{truthlemma1} we have the following biconditional:
\begin{equation}\label{1}
Ext(B) = \{ m/h \mid \mathcal{M}, m/h \vDash B \}.
\end{equation}

Now, assume that $I_aB \in w$. Then, by (A5), we also have
$[c]_aI_aB \in w$ and $\neg SI_aB \in w$. Take any $h_{w'}$ such
that $h_{w'} \in Choice^a_X(h_w)$. In the context of $\mathcal{M}$
this means that $w\simeq_a w'$. Therefore, we have $I_aB \in w'$.
By definition of $N_a$, this means that $Ext(B) \in
N_a(X/h_{w'})$. On the other hand, the fact that $\neg SI_aB \in
w$ means that the set
$$\gamma = \{ C \mid SC \in w \} \cup \{ \neg I_aB \}$$
is consistent. Indeed, otherwise we would have
$$\{ C \mid SC \in w \} \vdash I_aB,$$
and further, by standard $S5$ reasoning
$$\{ SC \mid SC \in w \} \vdash SI_aB,$$
and so, given, maxiconsistency of $w$, we would have $SI_aB \in
w$, contrary to our assumption. Therefore, consider an arbitrary
$w'' \in W$ extending $\gamma$. By definition, $w'' \in X$ so that
$h_{w''}$ goes through $X$, and we have $Ext(B) \notin
N_a(X/h_{w''})$ by definition of $N_a$.

Putting all this together, we get that, by \eqref{1}, $\{ m/h \mid
\mathcal{M}, m/h \vDash B \} \in N_a(X/h_{w'})$ for every $h_{w'}
\in Choice^a_X(h_w)$ and  $\{ m/h \mid \mathcal{M}, m/h \vDash B
\} \notin N_a(X/h_{w''})$ for some $h_{w'}$ going through $X$.
That is to say, we get that $\mathcal{M}, X/h_w \vDash I_aB$.

On the other hand, if $I_aB \notin w$, then, of course, $Ext(B)
\notin N_a(X/h_w)$, and given the fact that $h_w \in
Choice^a_X(h_w)$ and the biconditional \eqref{1}, we get that
$\mathcal{M}, X/h_w \not\vDash I_aB$ immediately.
\end{proof}

Now we are ready for our main result.

\begin{theorem}
Let $\Theta$ be a consistent set of sentences. Then $\Theta$
has a model.
\end{theorem}
\begin{proof}
Consider any maxiconsistent set $w$ extending $\Theta$ and its
corresponding $R$-equivalence class $X$. Then, by Lemma
\ref{truthlemma2}, we have $\mathcal{M}, X/h_w \vDash \Theta$.
\end{proof}
We also get \textbf{compactness} of $L$ as a standard consequence
of strong completeness.

 }
\end{document}